\documentclass[11pt,leqno]{amsart}
\makeatletter
\usepackage{amsfonts,delarray,amssymb,amsmath,amsthm,a4,a4wide}
\usepackage{color}

\usepackage[margin=1in]{geometry} 

\title[Green's function of the linearized Monge-Amp\`ere operator ]{Remarks on the Green's function of the linearized Monge-Amp\`ere operator}
\author{Nam Q. Le}
\address{Department of Mathematics, Indiana University, 831 E 3rd St,
Bloomington, IN 47405, USA}
\email{nqle@indiana.edu}
\newcommand{\review}[2][\right]{\relax
\ifx#1\right\relax \left.\fi#2#1\rvert}

\newtheorem{theorem}{Theorem}[section]

\newtheorem{propo}[theorem]{Proposition}
\newtheorem{remark}[theorem]{Remark}
\newtheorem{cor}[theorem]{Corollary}
\newtheorem{lemma}[theorem]{Lemma}
\newcommand{\bef}{\begin{flushright}}
\newcommand{\eef}{\end{flushright}}

\newcommand{\eval}[2][\right]{\relax
\ifx#1\right\relax \left.\fi#2#1\rvert}

\numberwithin{equation}{section}

\newcommand{\p}{\partial}
\newcommand\e{\varepsilon}  
\newcommand{\h}{\hspace*{.24in}}

\def\h{\hspace*{.24in}}

\def\beq{\begin{eqnarray*}}
\def\eeq{\end{eqnarray*}}

\def\RR{\mbox{$I\hspace{-.06in}R$}}

\newenvironment{myindentpar}[1]%
{\begin{list}{}%
         {\setlength{\leftmargin}{#1}}%
         \item[]%
}
{\end{list}}
 
\begin{document}
\begin{abstract}
In this note, we obtain sharp bounds for the Green's function of the linearized Monge-Amp\`ere operators associated to convex functions with either 
Hessian determinant bounded away from zero and infinity or Monge-Amp\`ere measure satisfying a doubling condition. Our result is an affine invariant version of the 
classical result of Littman-Stampacchia-Weinberger for uniformly elliptic operators in divergence form. We also obtain the $L^{p}$ integrability for the gradient of the Green's function
in two dimensions. As an application, we obtain 
a removable singularity result for the linearized Monge-Amp\`ere equation.
\end{abstract}
\maketitle
\pagenumbering{arabic}
\section{Introduction and Statement of the main result}
In \cite{LSW}, Littman-Stampacchia-Weinberger established the fundamental sharp bounds for the Green's function of linear, uniformly elliptic operator in divergence form
$L= - \p_j(a^{ij} \p_i)$ on a smooth, bounded domain $V\subset\RR^n$.
Here the coefficient matrix $(a^{ij})$ is symmetric with real, bounded measurable entries and uniformly elliptic, that is, there are positive constants $\lambda, \Lambda$ such that
\begin{equation}\lambda I_n \leq (a^{ij})\leq \Lambda I_n.
 \label{Eeq}
\end{equation}
This condition is invariant under the orthogonal transformation of coordinates. Let $g(x, y)$ be the Green's function of the operator L on $V$, that is, for 
each $y\in V$, $g(\cdot, y)$ is a positive solution of 
$$Lg(\cdot, y)=\delta_{y}~\text{in} ~V,~\text{and}~
 g(\cdot, y) =0~\text{on}~\p V.$$
Then, it was shown in \cite{LSW} that  $g$ is comparable to the Green's function of the Laplace operator $-\Delta$. In particular, $g$
satisfies the following sharp bounds in dimensions $n\geq 3$:
\begin{equation}C^{-1}|x-y|^{-(n-2)}\leq g(x, y)\leq C |x-y|^{-(n-2)}~\forall y\in V
\label{GLSW}
\end{equation}
where $C=C(n,\lambda,\Lambda, V, dist(y, \p V)).$ Other important properties of $g$ such as integrability and continuity of its gradient were studied by
Gr\"uter-Widman in \cite{GW}.

This note is concerned with estimates, analogous to (\ref{GLSW}), for the Green's function of the linearized Monge-Amp\`ere equation, an affine invariant version of (\ref{Eeq}). Let $\Omega$ be a bounded, smooth, uniformly convex domain in $\RR^n$
and $\mu$ a Borel measure in $\Omega$ with $\mu(\Omega)<\infty$. Let $u$ be a convex function 
satisfying the following Monge-Amp\`ere equation in the sense of Aleksandrov (see \cite{G})
\begin{equation}\det D^{2} u=\mu\h \text{in} ~\Omega,~u=0 ~\text{on} ~\p \Omega.
\label{MAeq}
\end{equation}

We consider two typical cases. The first case is when $\mu = f dx$ where $f$ is bounded from below and above by some  positive constants $\lambda, \Lambda$:
\begin{equation}\lambda\leq f\leq\Lambda\h \text{in} ~\Omega.
 \label{fbound}
\end{equation}
The second case is when $\mu$ is {\it doubling with respect to the center of mass.} This will be made more precise later. We assume throughout the note that $u$ is smooth but our estimates do not
depend on the smoothness of $u$.

Denote by $U= (U^{ij})\equiv (\det D^2 u) (D^2 u)^{-1}$ the cofactor matrix of the Hessian matrix $D^2u$.
Then, the linearized operator of the Monge-Amp\`ere equation (\ref{MAeq}) is given by
$$L_u v:= -U^{ij}v_{ij}\equiv - (U^{ij} v)_{ij}.$$
The last equation is due to the fact that $U = (U^{ij})$ is divergence-free. The reader is referred to \cite{CG, TW} and the references therein for more information
on the theory of linearized Monge-Amp\`ere equation and its applications to fluid mechanics and differential geometry.

The Monge-Amp\`ere and linearized Monge-Amp\`ere equations are invariant under unimodular transformation of coordinates. Indeed, let $T$ be a linear transformation with $\det T=1$. Then, the rescaled functions
$$\tilde u(x) = u(Tx),~\tilde v(x) = v (Tx),$$
satisfy
$$\det D^2 \tilde u(x) =\det D^2 u(x),~\tilde U^{ij} \tilde v_{ij}(x) = U^{ij}v_{ij}(Tx).$$

The linearized Monge-Amp\`ere operator $L_u$ is in general not uniformly elliptic. Under (\ref{MAeq}) and (\ref{fbound}), the eigenvalues of $U = (U^{ij})$ are not necessarily 
bounded away from $0$ and $\infty.$ The degeneracy is the main difficulty in establishing our affine invariant analogue of (\ref{GLSW}). As in \cite{CG}, 
we handle the degeneracy of $L_u$ by working with sections of solutions to the Monge-Amp\`ere equations. These sections have the same role as Euclidean balls have in the classical theory. The section of $u$ with center $x_0$ and height $t$ is defined by
$$S_u(x_0, t)=\{x\in\overline{\Omega}: u(x) <u(x_0) + \nabla u(x_0) (x-x_0) + t\}.$$
We say that the Borel measure $\mu$ is {\it doubling with respect to the center of mass} on the sections of $u$ if there exist constants $\beta>1$ and $0<\alpha<1$ such that for all sections 
$S_u (x_0, t)$,
\begin{equation}
 \label{muDC}
 \mu (S_u(x_0, t)) \leq \beta \mu (\alpha S_u (x_0, t/2)).
\end{equation}
Here $\alpha S_u (x_0, t)$ denotes the $\alpha$-dilation of $S_u(x_0, t)$ with respect to its center of mass $x^{\ast}$:
$$\alpha S_u(x_0, t)= \{x^{\ast} + \alpha (x- x^{\ast}): x\in S_u(x_0, t)\}.$$

Let $g_{V}(x, y)$ be the Green's function of $L_u$ in $V$ where $V\subset\subset\Omega$.  
\subsection{The main result}
In this note, we obtain the sharp upper bounds for $g_V$ in all dimensions when $u$ satisfies (\ref{MAeq}) and (\ref{fbound}). We also obtain
the sharp lower bounds for $g_V$ when $\mu$ satisfies a more general doubling condition (\ref{muDC}).
Our main result states:
\begin{theorem} \label{Gthm} Fix $x_0\in V$. Suppose that $0<t<1/4$, $S_u(x_0, 2t)\subset\subset V$ if $n\geq 3$ and $S_u(x_0, t^{1/2})\subset\subset V$ if $n=2$. 
\begin{myindentpar}{1cm}
(i) Assume that (\ref{MAeq}) and (\ref{fbound}) are satisfied.  
Then, for $x\in S_u(x_0, t)$, we have
$$g_V(x, x_0) \geq \begin{cases} c(n,\lambda, \Lambda) t^{-\frac{n-2}{2}} &\mbox{if } n \geq 3 \\
c(n,\lambda, \Lambda) |\text{log}~ t| & \mbox{if } n = 2. \end{cases} $$
Moreover, for $x\in \p S_u(x_0, t)$, we have
$$g_V(x, x_0) \leq \begin{cases} C(V,\Omega, n,\lambda, \Lambda) t^{-\frac{n-2}{2}} &\mbox{if } n \geq 3 \\
C(V, \Omega, n,\lambda, \Lambda) |\text{log}~ t| & \mbox{if } n = 2. \end{cases} $$
(ii) Assume that (\ref{MAeq}) and (\ref{muDC}) are satisfied.  Then, for $x\in S_u(x_0, t)$, we have
$$\displaystyle g_V(x, x_0) \geq \begin{cases} c(n,\alpha, \beta) t \left(\mu(S_u(x_0, t))\right)^{-1} &\mbox{if } n \geq 3 \\
\displaystyle c(n,\alpha,\beta) \frac{|\text{log}~ t|^2}{\int_t^{t^{1/2}} \frac{ \mu(S_u(x_0, s)) ds}{s^2}} & \mbox{if } n = 2. \end{cases} $$
(iii) Suppose that $n=2$ and (\ref{MAeq}) and (\ref{fbound}) are satisfied. Then there exists $p_{\ast}(n, \lambda,\Lambda)>1$ such that for all $1<p< p_{\ast}$ and all 
$S_u(x_0, r^{1/2})\subset\subset V$, we have
$$\left(\int_{S_u(x_0, r)}|\nabla g_V (x, x_0)|^p dx\right)^{\frac{1}{p}}\leq C(V, \Omega, n,p, \lambda,\Lambda, r).$$
\end{myindentpar}
\end{theorem}
Our estimates in Theorem \ref{Gthm} depend only on the dimension, the upper and lower bound of the Hessian determinant. They do not depend on the bounds on eigenvalues of the Hessian matrix $D^2u$. Properties of the 
Green's function $g_V$ have played an important role in establishing Sobolev inequality for the Monge-Amp\`ere quasi-metric structure \cite{TiW, MRL}.
\begin{remark} 
In Theorem \ref{Gthm} (iii), we can choose
$$p_{\ast} = 1 + \frac{\e}{2+\e}$$
where $\e= \e(n,\lambda,\Lambda)$ comes from 
De Philippis-Figalli-Savin
and Schmidt's $W^{2, 1+\e}$ estimates \cite{DPFS, Sch} for the Monge-Amp\`ere equation satisfying (\ref{MAeq}) and (\ref{fbound}). Thus 
$p_{\ast}\rightarrow 2$ when $\e\rightarrow \infty$. Hence, by Caffarelli's $W^{2,p}$ estimates for the Monge-Amp\`ere equations \cite{C}, we can take $p^{\ast}=2$ when
$f$ is continuous.
\end{remark}

\begin{remark}
In the case of Green's function of uniformly elliptic operators, Theorem \ref{Gthm} (iii) with all $p<2$ is attributed to Stampacchia. In higher dimensions, Gr\"uter and Widman 
\cite{GW} proved
the $L^{p}$ integrability of the gradient of the Green's function for all $p<\frac{n}{n-1}$. It would be interesting to prove the $L^{p}$ integrability for some $p>1$ 
for the gradient of the Green's function of the linearized Monge-Amp\`ere operator in dimensions $n\geq 3$.
\end{remark}

As a corollary, we use the sharp lower bound for the Green's function in Theorem \ref{Gthm} to prove a 
removable singularity result for the linearized Monge-Amp\`ere equation.

\begin{cor} Assume that $V\subset\subset \Omega$ and
$\lambda\leq \det D^2 u\leq\Lambda~\text{in}~\Omega.$
Suppose that a function $v$ solves $U^{ij}v_{ij}=0$ in $S_u(0, R)\backslash \{0\}\subset V$ and satisfies 
$$|v(x)|= \begin{cases} o(r^{\frac{2-n}{2}}) &\mbox{if } n \geq 3 \\
o( |\text{log} ~r|) & \mbox{if } n = 2 \end{cases}~\text{on}~\p S_u(0, r)~\text{as}~r\rightarrow 0.$$ Then v has a removable singularity at $0.$
\label{RemoveSing}
\end{cor}
\subsection{Previous results} Various properties of the Green's function of the linearized Monge-Amp\`ere operator $L_u$ under different conditions on $\mu$ have been
studied by several authors, including Tian-Wang \cite{TiW} and Maldonado \cite{MRL}.
Tian-Wang \cite{TiW} proved a decay estimate for the distribution function of $g_V$ under an 
$(A_{\infty})$ weight 
condition on $\mu$ (called $\bf (CG)$ there) and certain conditions on the size of sections of $u$. 
\begin{propo}(\cite[Lemma 3.3]{TiW}) \label{TiWlem}
Assume that $\mu$ satisfies the structure condition:\\
{\bf CG}. For any given $\e > 0$, there exists $\delta>0$ such that for any convex set $S\subset\Omega$ and
any set $E\subset S$, if $|E| \leq \delta|S|$, then
$\mu(E)\leq \e\mu(S)$
where $|\cdot|$ denotes the Lebesgue measure. Suppose that for any section $S_u(x, h)\subset \Omega$ of $u$, we have
$$C_1 |S_u(x, h)|^{1+\theta}\leq \mu(S_u(x, h)) \leq C_2 |S_u(x, h)|^{\frac{1}{n-1} +\sigma},$$
where $\theta\geq 0, C_1, C_2,\sigma>0$ are constants. Then, for any $y\in V$,
$$\mu \{x\in V: g_V(x, y)>t\}\leq K t^{-\frac{n(1+\theta)}{(n-1)(1+\theta)-1}}.$$
\end{propo}
When $\mu$ satisfies (\ref{muDC}) only, and $V= S_u(x, t)$, Maldonado \cite{MRL} obtained a similar result on the decay estimate for the distribution function of $g_V$.
His result can be stated as follows.
\begin{propo}(\cite[Theorem 3]{MRL})
 \label{Mthm} Suppose $V= S_u(x, t)\subset\subset \Omega$.
 There exists a constant $K_1$ depending only on $n,\alpha,\beta$ such that for all $z\in S_u(x,t/2)$, we have
 $$\mu(\{y\in V: g_V(y, z)>T\}) \leq K_1 (\mu(S_u(x, t)))^{-\frac{1}{n-1}} t^{\frac{n}{n-1}} T^{-\frac{n}{n-1}}~\forall T>0.$$
\end{propo}
\begin{remark} 
\label{TiWrm}
\begin{myindentpar}{1cm}
1. If $u$ satisfies (\ref{MAeq}) and (\ref{fbound}), then
in dimensions $n\geq 3$, Proposition \ref{TiWlem} gives a sharp upper bound for $g_V$. In particular, for small $t$ and $x\in\p S_u(x_0, t)$, we have
$$g_V(x, x_0)\leq Ct^{-\frac{n-2}{2}}.$$
2. If $u$ satisfies (\ref{MAeq}) and (\ref{muDC}), then Proposition \ref{Mthm} gives a sharp upper bound for $g_V$ in dimensions $n\geq 3$ when V is a section of $u$. When $V= S_u(x_0, t)$,  we have
$$g_V(x, x_0) \leq K_1^{\frac{n-1}{n}} t[\mu(S_u(x_0, t)]^{-\frac{1}{n}} [\mu(S_u(x_0, s))]^{-\frac{n-1}{n}}~\forall~x\in\p S_u(x_0, s)~(0<s<t).$$
In particular, by Lemma \ref{muDP}, we have
$$g_V(x, x_0) \leq C(K_1,\alpha,\beta) t[\mu(S_u(x_0, t)]^{-1}~\forall~x\in\p S_u(x_0, t/2).$$
\end{myindentpar}
\end{remark}
For reader's convenient, we will prove the estimates in this remark in Section \ref{sec_n3}. 
\vglue 0.2cm

The proof of (\ref{GLSW}) in \cite{LSW} was based on potential theory employing capacity and the fundamental result of 
De Giorgi-Nash-Moser on H\"older continuity of solutions of uniformly elliptic equations in divergence form. Our proof of Theorem \ref{Gthm}(i) is based on the 
fundamental result of Caffarelli-Guti\'errez \cite{CG} on H\"older continuity of solutions of the linearized Monge-Amp\`ere equation. We find a direct argument for Theorem \ref{Gthm}(i) without using capacity; see Section \ref{sec_n3}. We also provide an alternate proof for the lower bound of the Green's function in Theorem \ref{Gthm} using 
capacity; see Section \ref{sec_n2}. This potential theoretic approach works for general doubling Monge-Amp\`ere measures, thus allowing us to prove  Theorem \ref{Gthm}(ii); one of the key
ingredients here is Maldonado's Harnack inequality \cite{MCV} for linearized Monge-Amnp\`ere equations under a doubling condition. The proof of Theorem \ref{Gthm} (iii) makes use of De Philippis-Figalli-Savin
and Schmidt's $W^{2, 1+\e}$ estimates \cite{DPFS, Sch} for the Monge-Amp\`ere equation that are valid for all dimensions and the $L^q$ integrability of the Green's function for all finite $q$ in two
dimensions.
\section{Preliminaries}
Throughout, we denote by $c$, $C$ positive constants depending on $\lambda$, $\Lambda$, 
$n$, $\alpha, \beta$, and their values may change from line to line whenever 
there is no possibility of confusion. We refer to such constants as {\it universal constants}.
\subsection{Monge-Amp\`ere measure bounded away from $0$ and $\infty$}
In this section, we assume that
$$\lambda\leq \det D^2 u\leq\Lambda~\text{in}~\Omega.$$
Throughout, we use the following volume growth for compactly supported sections:
\begin{lemma} \label{vol_lem} If $S_u(x, t)\subset\subset \Omega$ then
$$c_1(n,\lambda,\Lambda)t^{\frac{n}{2}}\leq |S_u(x, t)|\leq C_1(n,\lambda,\Lambda)t^{\frac{n}{2}}.$$
\end{lemma}
The Caffarelli-Guti\'errez's Harnack inequality for the linearized Monge-Amp\`ere equation states:
\begin{theorem}(\cite{CG}) \label{Holder_thm}
For each compactly supported section $S_u (x, t)\subset\subset\Omega$, and any nonnegative solution $v$ of $L_u v=0$ in $S_u (x, t)$, we have
$$\sup_{S_u(x,\tau t)} v\leq C \inf_{S_u(x, \tau t)} v$$
for universal $\tau, C.$
\end{theorem}
Since the 
linearized Monge-Amp\`ere operator $L_u v$ can be written in both divergence form and non-divergence form,  Caffarelli-Guti\'errez's theorem is the affine 
invariant analogue of De Giorgi-Nash-Moser's theorem and also Krylov-Safonov's theorem on H\"older continuity of solutions of uniformly elliptic equations in nondivergence 
form. Theorem \ref{Holder_thm} will play an important role in our proof of the main result.

We also need the following Vitali type covering lemma.
\begin{lemma}[Vitali covering, \cite{DPFS}] Let $D$ be a compact set in $\Omega$ and assume that to each $x\in D$ we associate a corresponding section $S_u(x, h)\subset\subset\Omega$. Then we can find a finite number of these sections $S_u(x_i, h_i), i=1,\cdots, m,$ such that
$$D \subset \bigcup_{i=1}^m S_u(x_i, h_i),~\text{with}~ S_u(x_i, \delta h_i)~\text{disjoint},$$
where $\delta>0$ is a small constant that depends only on $\lambda$, $\Lambda$ and $n$.
\label{cov_lem}
\end{lemma}
\subsection{ Monge-Amp\`ere measure satisfying a doubling condition}
In this section, we assume that 
$\det D^2 u=\mu$
where $\mu$ satisfies (\ref{muDC}). Then $\mu$ is {\it doubling with respect to the parameter} on the sections of $u$:
\begin{lemma}\cite[Corollary 3.3.2]{G}  \label{muDP}
If $S_u(x, 2t)\subset\subset \Omega$ then
there is a constant $\beta'$ depending only on $n, \beta$ and $\alpha$ such that
$$\mu(S_u(x, 2t))\leq \beta^{'}\mu(S_u(x, t)).$$
\end{lemma}

Maldonado \cite{MCV}, extending the work of Caffarelli-Guti\'errez,  proved the following Harnack inequality for the linearized Monge-Amp\`ere under minimal geometric condition, namely, the doubling condition (\ref{muDC}).
\begin{theorem}(\cite{MCV})
For each compactly supported section $S_u (x, t)\subset\subset\Omega$, and any nonnegative solution $v$ of $L_u v=0$ in $S_u (x, t)$, we have for 
$$\sup_{S_u(x,\tau t)} v\leq C \inf_{S_u(x, \tau t)} v$$
for universal $\tau, C$ depending only on $n, \beta$ and $\alpha$.
\label{MHolder_thm}
\end{theorem}
\section{Bounding the Green's function}
\label{sec_n3}
In this section, we prove Theorem \ref{Gthm}(i) and (iii) and Corollary \ref{RemoveSing}. Assume throughout this section that (\ref{MAeq}) and (\ref{fbound}) are satisfied. 

The proof of Theorem \ref{Gthm}(i) relies on 
three Lemmas \ref{smallV}, \ref{max_lem} and \ref{small_lem}. Lemma \ref{smallV} gives the bounds for the Green's function $g_V(x, x_0)$ 
in the special case where $V$ is itself a section of $u$ centered at $x_0$. Lemma \ref{max_lem} estimates how the maximum of $g_V(x, x_0)$ on a section of
$u$ centered at $x_0$ changes when we pass to a concentric section with double height. Lemma \ref{small_lem} gives the upper bound for $g_V$ near $\p V$.
\begin{lemma}\label{smallV}
If $V= S_u(x_0, t)$ then
$$g_V(x, x_0)\geq c(n,\lambda,\Lambda)t^{-\frac{n-2}{2}}~\forall x\in S_{u}(x_0, t/2)$$
and
$$g_V(x, x_0)\leq C(n,\lambda,\Lambda)t^{-\frac{n-2}{2}}~\forall x\in \p S_{u}(x_0, t/2).$$
\end{lemma}
\begin{lemma}\label{max_lem}
If $S_u(x_0, 2t)\subset\subset V$, then 
\begin{equation}\max_{x\in\p S_u(x_0, t)} g_V(x, x_0)\leq Ct^{-\frac{n-2}{2}} + \max_{z\in \p S_u(x_0, 2t)} g_V(z, x_0).
\label{iter_max}
\end{equation}
\end{lemma}
In the next lemma, by considering the Green's function on a larger domain containing $V$, we assume that $dist(x_0, \partial V)\geq dist(V,\partial\Omega)$ for the purpose of 
obtaining an upper bound for
$g_V(x_0, \cdot)$.
\begin{lemma}\label{small_lem} 
There exist constants $r(V,\Omega, n,\lambda,\Lambda)$ and $C(V,\Omega, n,\lambda,\Lambda)$ such that
\begin{equation}
S_u(x_0, 2r)\subset\subset V~\text{and}~ \max_{x\in\p S_u (x_0, r)} g_V(x, x_0)\leq C(V,\Omega, n,\lambda,\Lambda).
\label{smallr}
\end{equation}
\end{lemma}
\begin{proof}[Proof of Theorem \ref{Gthm}]
{\it Part (i).}
We will prove the lower and upper bound for $g_V$.\\
{\bf Lower bound for $g_V$.} Consider the following cases.\\
{\bf Case 1:} $n\geq 3$ and $S_u(x_0, 2t)\subset\subset V$. In this case, the difference $w:= g_V(x, x_0)-g_{S_u(x_0, 2t)}(x, x_0)$ solves
$$U^{ij} w_{ij}=0~\text{in}~ S_{u}(x_0, 2t),
\text{ with }
w>0~\text{on}~\p S_u(x_0, 2t).$$
Thus, by the maximum principle, $w(x)\geq 0$ for $x\in S_u(x_0, t)$. It follows from Lemma \ref{smallV} that
$$g_V(x, x_0)\geq g_{S_u(x_0, 2t)}(x, x_0)\geq c(n,\lambda, \Lambda)t^{-\frac{n-2}{2}}~\forall x\in S_u(x_0, t).$$
{\bf Case 2:} $n=2$ and $S_u(x_0, t^{1/2})\subset\subset V$. Suppose that $S_u(x_0, 2h)\subset\subset V$. Then, the function
$$w(x) = g_V(x, x_0)-\inf_{y\in \p S_u(x_0, 2h)} g_V(y, x_0)- g_{S_u(x_0, 2h)}(x, x_0)$$
satisfies
$$L_u w=0~\text{in}~ S_u(x_0, 2h)~\text{with}~w\geq 0~\text{on}~\p S_u(x_0, 2h).$$
By the maximum principle, we have $w\geq 0$ in $ S_u(x_0, 2h)$. Thus, by Lemma \ref{smallV}, we find that 
\begin{equation}g_V(x, x_0)-\inf_{y\in \p S_u(x_0, 2h)} g_V(y, x_0)\geq g_{S_u(x_0, 2h)}(x, x_0)\geq c~\forall ~x\in S_u(x_0, h).
\label{lower2Dh}
\end{equation}
Choose an integer $k\geq 1$ such that
$2^k \leq t^{-1/2}< 2^{k+1}.$
Then
$$|\text{log t}|\leq Ck~\text{and}~ 2^k t\leq t^{1/2}.$$
Applying (\ref{lower2Dh}) to $h=t, 2t, \cdots, 2^{k-1} t$, we get
\begin{eqnarray*}\inf_{y\in \p S_u(x_0, t)} g_V(y, x_0) \geq \inf_{y\in \p S_u(x_0, 2t)} g_V(y, x_0) + c \geq \cdots &\geq& \inf_{y\in \p S_u(x_0, 2^kt)} g_V(y, x_0) + kc\\ &\geq& kc \geq 
c |\text{log t}|.
\end{eqnarray*}
{\bf Upper bound for $g_V$.} Our proof of the upper bound for $g_V$ just follows from iterating the estimate in Lemma \ref{max_lem} and the upper bound for $g_V$ near $\p V$ in 
Lemma \ref{small_lem}.\\
{\it Part (iii).} Recall that in this part $n=2$.
 Let $v(x)= g_V(x, x_0)$ and $S=S_u(x_0, r)$. Then the upper bound for $v$ in Theorem \ref{Gthm}(i) implies that $v\in L^q(S)$ for all $q<\infty$ with the bound
 \begin{equation}\|v\|_{L^q(S)}\leq C(V,\Omega, \lambda,\Lambda, q, r).
  \label{Glq}
 \end{equation}

By \cite[Theorem 6.2]{MCV}, we have
\begin{equation}\int_{S} U^{ij} v_{i}(x) v_{j}(x) \frac{1}{v(x)^2} dx \leq C(n,\lambda,\lambda) \frac{\mu(S)}{r}\leq  C(n,\lambda,\lambda)
 \label{log_ineq}
\end{equation}
where we used the upper bound on volume of section in Lemma \ref{vol_lem} in the last inequality. Next, we use the following inequality
$$U^{ij} v_{i}(x) v_{j}(x)  \geq \frac{\det D^2 u |\nabla v|^2}{\Delta u}$$
whose simple proof can be found in \cite[Lemma 2.1]{CG}.
Thus, for all integrable function $f$ we have
\begin{eqnarray*}|\nabla v|^2 |f|^2 = (\Delta u |f|^2) \frac{|\nabla v|^2}{\Delta u}&\leq& \frac{1}{\lambda }(\Delta u |f|^2) \frac{\det D^2u|\nabla v|^2}{\Delta u}\\ & \leq& \frac{1}{\lambda} 
 (\Delta u |f|^2 v^2) U^{ij} v_{i}(x) v_{j}(x) \frac{1}{v(x)^2}.
 \end{eqnarray*}
Integrating over $S$ and using Cauchy-Schwartz inequality and (\ref{log_ineq}), one finds
\begin{eqnarray}
 \int_{S} |\nabla v||f|&\leq& \frac{1}{\sqrt{\lambda}}\left(\int_S U^{ij} v_{i}(x) v_{j}(x) \frac{1}{v(x)^2} \right)^{1/2}\left(\int_S \Delta u |f|^2 v^2\right)^{1/2} 
 \nonumber\\&\leq& C(n,\lambda,\Lambda)\left(\int_S \Delta u |f|^2 v^2\right)^{1/2}.
 \label{dgr}
\end{eqnarray}
By the  De Philippis-Figalli-Savin and Schmidt's $W^{2, 1+\e}$ estimates for the Monge-Amp\`ere equation \cite{DPFS, Sch}, there exists $\e = \e(n,\lambda,\Lambda)>0$ such that
$D^2 u\in L^{1+\e}_{loc}(\Omega)$. Thus, by H\"older inequality, 
\begin{eqnarray}
 \int_S \Delta u |f|^2 v^2 &\leq& \left(\int_S (\Delta u)^{1+\e} \right)^{\frac{1}{1+\e}} \left(\int_S  |f|^{\frac{2(1+\e)}{\e}} v^{\frac{2(1+\e)}{\e}}\right)^{\frac{\e}{1+\e}}
 \nonumber \\ &\leq & C(n,\lambda,\Lambda, r) \left(\int_S  |f|^{\frac{2(1+\e)}{\e}} v^{\frac{2(1+\e)}{\e}}\right)^{\frac{\e}{1+\e}}.
 \label{higher-int}
\end{eqnarray}
From (\ref{Glq}), we find that $\left(\int_S  |f|^{\frac{2(1+\e)}{\e}} v^{\frac{2(1+\e)}{\e}}\right)^{\frac{\e}{1+\e}}$ is finite if
$f\in L^{\frac{p}{p-1}}(S)$ where $\frac{p}{p-1}> \frac{2(1+\e)}{\e}$, or
$$p<1 + \frac{\e}{2+ \e}:= p_{\ast}.$$
Combining (\ref{dgr}) with (\ref{higher-int}) and (\ref{Glq}), one finds that
$$\int_{S} |\nabla v||f| \leq C(V,\Omega, n,p,\lambda,\Lambda, r) \|f\|_{L^{\frac{p}{p-1}}(S)}$$
for all $f\in L^{\frac{p}{p-1}}(S)$ where
$1<p<p_{\ast}.$ Theorem \ref{Gthm} (iii) then follows from duality.
\end{proof}

\begin{proof}[Proof of Corollary \ref{RemoveSing}]
Let $\tilde v$ solves 
\begin{equation*}
 \left\{
 \begin{alignedat}{2}
   U^{ij}\tilde v_{ij}~&=0\h~&&\text{in} ~S_u(0, R), \\\
\tilde v &=v\h~&&\text{on}~\p S_u(0, R).
 \end{alignedat}
 \right.
\end{equation*}
We will prove that $v=\tilde v$ in $S_u(0, R)\backslash \{0\}$. We only consider the case $n\geq 3$. The case $n=2$ is similar. Let $w=\tilde v- v$ in $S_u(0, R)\backslash \{0\}$ and $M_r=\max_{\p S_u(0, r)} |w|.$ Let $\sigma(x)= g_{S_u(0, R)}(x, 0)$. By the lower bound for the Green's function in Theorem \ref{Gthm}, it is obvious that
$$|w(x)|\leq CM_r r^{\frac{n-2}{2}} \sigma(x)~\text{on}~ S_u(0, r).$$
Note that
$$U^{ij}(w-CM_r r^{\frac{n-2}{2}} \sigma (x) )_{ij}=0~\text{in}~ S_u(0, R)\backslash S_u(0, r).$$
Thus, by the maximum principle in $S_u(0, R)\backslash S_u(0, r)$, we have
$$|w(x)|\leq CM_r r^{\frac{n-2}{2}} \sigma (x)~\text{in}~S_u(0, R)\backslash S_u(0, r).$$
Observe that
$$M_r= \max_{\p S_u(0, r)} |v-\tilde v|\leq M + \max_{\p S_u(0, r)} |v|$$
where $M=\max_{\p S_u(0, R)} |\tilde v|$. For each fixed $x\neq 0$, we can choose $r$ small so that $x\not\in S_u(0, r)$ and hence, by our hypothesis on the asymptotic behavior of $v$ near $0$,
$$|w(x)|\leq CM r^{\frac{n-2}{2}} \sigma(x) + C \sigma(x) r^{\frac{n-2}{2}}\max_{\p S_u(0, r)} |v|\rightarrow 0~\text{as}~r\rightarrow 0. $$
This proves $v=\tilde v$ in $S_u(0, R)\backslash \{0\}$.
\end{proof}
\begin{proof}[Proof of Lemma \ref{smallV}]
By subtracting a linear function, we can assume that
$u\geq 0~\text{and}~ u(x_0)=0.$
For simplicity, let us  denote $\sigma(x) =g_V(x, x_0)$. Then on $V= S_u(x_0, t)$, $\sigma$ satisfies
\begin{equation}
 \left\{
 \begin{alignedat}{2}
  L_u\sigma~&=\delta_{x_0} \h~&&\text{in} ~V, \\\
\sigma &=0\h~&&\text{on}~\p V.
 \end{alignedat}
\right.
\label{DMA}
\end{equation}
Multiplying both sides of (\ref{DMA}) by $u(x)-t$ and 
integrating by parts twice, we get
$$-t = u(x_0)-t=\int_{V} (L_u \sigma) (u-t) =\int_{V}-U^{ij}\sigma_{ij} (u-t)=\int_{V}-U^{ij}\sigma u_{ij}=\int_{V}-n f \sigma.$$
The bounds on $f$ then give the following bounds for the integral of $\sigma$:
$$\frac{t}{n\lambda}\geq \int_V \sigma \geq \frac{t}{n\Lambda}.$$
On the other hand, by the ABP estimate, for any $\varphi\in L^n(V)$, the solution $\psi$ to 
$$-U^{ij}\psi_{ij}=\varphi~\text{in}~V,~\psi=0~\text{on}~\p V,$$
satisfies
$$|\int_V \sigma(x) \varphi(x) dx|=|\psi(x_0)|\leq C(n)|V|^{1/n}\left\|\frac{\varphi}{\det U}\right\|_{L^n(V)} \leq C(n,\lambda, \Lambda) |V|^{1/n} \|\varphi\|_{L^n(V)}.$$
Here we used the identity $\det U= (\det D^2 u)^{n-1}$. By duality, we obtain
$$\left(\int_V \sigma^{\frac{n}{n-1}} \right)^{\frac{n-1}{n}}\leq C(n,\lambda,\Lambda) |V|^{1/n}.$$
This is essentially  inequality (2.3) in \cite{MRL}.
Hence, by Lemma \ref{vol_lem},
$$\|\sigma\|_{L^{\frac{n}{n-1}}(S_u (x_0, t))}\leq C(n,\lambda,\Lambda) t^{1/2}.$$
Let $$K= (S_u(x_0, t)\backslash S_u (x_0, r_2 t)) \cup S_u (x_0, r_1 t)$$ where $0<r_1<1/2<r_2<1$. Then, by \cite[Lemma 6. 5. 1]{G} and Lemma \ref{vol_lem}, we can estimate
$$|K|\leq n(1-r_2)|S_u (x_0, t)| + |S_u(x_0, r_1 t)|\leq C_1 n(1-r_2)t^{n/2} + C_1(r_1 t)^{n/2}\leq \e^n t^{n/2}$$ 
for $$\e=\min\{ \frac{1}{2C_1 (n,\lambda,\Lambda)n\Lambda}, \left(\frac{1}{2c_1}\right)^{1/n}\}$$
if
$r_1, 1-r_2$ are universally small. Then by Lemma \ref{vol_lem},
\begin{equation}
\frac{c_1}{2} t^{n/2}\leq |S_u(x_0, t)\backslash K|\leq C_1 t^{n/2}.
\label{notK}
\end{equation}
On the other hand, by Holder inequality, we have
$$\int_K \sigma \leq \|\sigma\|_{L^{\frac{n}{n-1}}(K)}|K|^{1/n}\leq C(n,\lambda,\Lambda)t^{1/2} \e t^{1/2}= \frac{t}{2n\Lambda}.$$ It follows that
\begin{equation}\frac{t}{n\lambda}\geq\int_{S_u(x_0, t)\backslash K}\sigma\geq \frac{t}{2n\Lambda}.
\label{int_sig}
\end{equation} 
Given $0<r_1<r_2<1$ as above, 
we have
\begin{equation}\sup_{S_u(x_0, t)\backslash K} \sigma \leq C(n,\lambda,\Lambda) \inf_{S_u(x_0, t)\backslash K} \sigma.
\label{CGK}
\end{equation}
Combining (\ref{notK})-(\ref{CGK}), we find that
$$C^{-1}(n,\lambda,\Lambda)t^{-\frac{n-2}{2}}\leq\sigma(x) \leq C(n,\lambda,\Lambda)t^{-\frac{n-2}{2}}~\forall x\in S_u(x_0, t)\backslash K.$$
This line of argument is very similar to the proof of Lemma 5.1 in \cite{LS}. Since $r_2>1/2>r_1$, we obtain the desired upper bound for $\sigma(x)= g_V(x, x_0)$ when $x\in\p S_u(x_0, t/2)$ while from the maximum principle, we obtain the desired lower bound for $\sigma(x)= g_V(x, x_0)$ when $x\in S_u(x_0, t/2).$

For completeness, we include the details of (\ref{CGK}). By \cite[Theorem 3.3.10]{G}, we can find a universal $\alpha\in (0, 1)$ such that for each $x\in S_u(x_0, t)\backslash K$, the section $S_u(x, \alpha t)$ satisfies
$$x_0\not\in S_u(x,\alpha t)~\text{and}~ S_u(x, \alpha t)\subset S_u (x_0, t).$$ Using Lemma \ref{cov_lem}, we can find a collection of sections $S_u(x_i, \tau\alpha t)$ with $x_i\in S_u(x_0, t)\backslash K$ such that
$$S_u(x_0, t)\backslash K \subset \bigcup_{i\in I} S_u(x_i,\tau\alpha t)$$
and
$S_u(x_i,\delta\tau\alpha t)$ are disjoint for some universal $\delta\in (0, 1).$ By using the volume estimates in Lemma \ref{vol_lem}, we find that $|I|$ is universally bounded. Now, we apply Theorem \ref{Holder_thm} to each $S_u (x_i, \alpha t)$ to obtain (\ref{CGK}).
\end{proof}
\begin{proof}[Proof of Lemma \ref{max_lem}]
To prove (\ref{iter_max}), we consider
$$w(x) = g_V(x, x_0)-\inf_{y\in \p S_u(x_0, 2t)} g_V(y, x_0)- g_{S_u(x_0, 2t)}(x, x_0).$$
It satisfies
$$L_u w=0~\text{in}~ S_u(x_0, 2t)~\text{with}~w\geq 0~\text{on}~\p S_u(x_0, 2t).$$
In $\overline{S_u(x_0, 2t)}$, $w$ attains its maximum value on the boundary $\p S_u(x_0, 2t)$. Thus, for $x\in \p S_u(x_0, t)$, we have
\begin{multline*}g_V(x, x_0)-\inf_{y\in \p S_u(x_0, 2t)} g_V(y, x_0)- g_{S_u(x_0, 2t)}(x, x_0)\leq \max_{z\in \p S_u(x_0, 2t)} w\\= \max_{z\in \p S_u(x_0, 2t)} g_V(z, x_0)-\inf_{y\in \p S_u(x_0, 2t)} g_V(y, x_0)
\end{multline*}
since $g_{S_u(x_0, 2t)}(x, x_0)=0$ on $\p S_u(x_0, 2t)$. This together with Lemma \ref{smallV} gives
\begin{eqnarray*}\max_{x\in \p S_u(x_0, t)} g_V(x, x_0)&\leq& \max_{z\in \p S_u(x_0, t)} g_{S_u(x_0, 2t)}(z, x_0) + \max_{z\in \p S_u(x_0, 2t)} g_V(z, x_0)\\ &\leq & Ct^{-\frac{n-2}{2}} + \max_{z\in \p S_u(x_0, 2t)} g_V(z, x_0).
\end{eqnarray*}
Therefore, (\ref{iter_max}) is proved.
\end{proof}
\begin{proof}[Proof of Lemma \ref{small_lem}]
The existence of $r(V,\Omega, n,\lambda,\Lambda)$ is easy to prove by the $C^{1,\alpha}$ estimate for $u$ which implies in particular that
$S_u(x_0, h)\subset B(x_0, Ch^{\alpha}).$ We now prove
$$ \max_{\p S_u (x_0, r)} g_V(x, x_0)\leq C(V,\Omega, n,\lambda,\Lambda).$$
To do this, we first multiply $\sigma(x):= g_V(x, x_0)$ to $L_u \Phi$ for various choices of $\Phi=\Phi(x, u(x), Du(x))$ and then integrate by parts. 
Let $\nu$ be the unit outer-normal vector field on $\p V$. Note that, on $\p V$, we have
$\nu =-\frac{D\sigma}{|D\sigma|}.$
Integrating by parts, we get
\begin{eqnarray}
\int_{V}(L_u \Phi)\sigma =\int_{V} -U^{ij}\sigma \Phi_{ij} &=&\int_{V}(U^{ij}\sigma)_i \Phi_j -\int_{\p V} U^{ij}\sigma \Phi_j \nu_i =\int_{V}-(U^{ij}\sigma)_{ij}\Phi +\int_{\p V} U^{ij}\sigma_i\Phi\nu_j\nonumber\\&=& \Phi(x_0, u(x_0), Du(x_0)) -\int_{V}U^{ij}\Phi \sigma_i \frac{\sigma_j}{|D\sigma|}\nonumber\\
&=&\Phi(x_0, u(x_0), Du(x_0))-\int_{V}\Phi \rho dS.
\label{adj_iden}
\end{eqnarray}
Here, we denote
$$\rho=U^{ij} \sigma_i \frac{\sigma_j}{|D\sigma|}.$$
First, we choose $\Phi\equiv 1$. Then (\ref{adj_iden}) gives
\begin{equation}\int_{\p V}\rho dS=1.
\label{rho1}
\end{equation}
Next, we choose $\Phi\equiv u.$ Then, since $$L_u u= -U^{ij} u_{ij}=-n\det D^2 u=-nf,$$ (\ref{adj_iden}) gives
$$\int_V nf g_V(x, x_0) dx = \int_{\p V} \rho u dS - u(x_0).$$
By Aleksandrov's maximum principle \cite[Theorem 1.4.2]{G}, we have
$$|u(x_0)|, \max_{x\in\p V}|u(x)|\leq C(V, \Omega, n,\lambda, \Lambda).$$ 
Combining these with (\ref{rho1}), we get
$$\int_{S_u(x_0, 2r)} g_V (x, x_0) dx\leq C(V, \Omega, n,\lambda, \Lambda).$$
Using the lower bound for volume of sections in Lemma \ref{vol_lem} and Caffarelli-Guti\'errez's Harnack inequality in Theorem \ref{Holder_thm}, we get the second inequality in (\ref{smallr}).
\end{proof}

If we choose $\Phi\equiv |x|^2$ in  (\ref{adj_iden}) then, 
since
$L_u \Phi =-2U^{ij}\delta_{ij} =-2 \text{trace}~ U,$
we get from (\ref{adj_iden}) that
$$\int_{V}-2 \text{trace}~ U\sigma=\int_{V}(L_u \Phi)\sigma= |x_0|^2-\int_V |x|^2\rho dS.$$
Thus, by (\ref{rho1}),
$$2\int_V \text{trace}~ U\sigma = \int_{\p V}|x|^2\rho dS -|x_0|^2\leq \max_{x\in\p V} |x|^2- |x_0|^2.$$
This combined with the lower bound of $\sigma$ in Theorem \ref{Gthm} gives the following Corollary.

\begin{cor}  Assume that $V\subset\subset \Omega$ and
$\lambda\leq \det D^2 u\leq\Lambda~\text{in}~\Omega.$ If $S_u(x_0, 2t)\subset\subset V$ or $S_u(x_0, t^{1/2})\subset\subset V$ when $n=2$ then, we have
\begin{equation}\int_{S_u(x_0, t)} \text{trace}~ U \leq  \begin{cases} C(n,\lambda,\Lambda) t^{\frac{n-2}{2}}\left( \max_{x\in\p V} |x|^2-|x_0|^2\right) &\mbox{if } n \geq 3 \\
C(\lambda,\Lambda) |\text{log}~t|^{-1}\left( \max_{x\in\p V} |x|^2-|x_0|^2\right) & \mbox{if } n = 2 \end{cases}.
\label{A2ineq}
\end{equation}
\label{A2lem}
\end{cor}

We end this section with the proof of Remark \ref{TiWrm}.
\begin{proof}[Proof of Remark \ref{TiWrm}]
{\bf 1}. In dimensions $n\geq 3$, we can establish the upper bound for $g_V$ by using Proposition \ref{TiWlem}. When $u$ satisfies (\ref{MAeq}) and (\ref{fbound}), this proposition 
says that
$$|\{x\in V: g_V(x, x_0)>T\}|< K(V,\Omega, n,\lambda,\Lambda)T^{-\frac{n}{n-2}}.$$
We show that for small $t$ and $x\in\p S_u(x_0, t)$
$$g_V(x, x_0)\leq (\frac{K}{c_1})^{\frac{n-2}{n}} t^{-\frac{n-2}{2}}$$
where $c_1$ is the constant in Lemma \ref{vol_lem}. Indeed, assume that for some $t>0$, we have
$$T=\max_{x\in \p S_u(x_0, t)}g_V(x, x_0)>(\frac{K}{c_1})^{\frac{n-2}{n}} t^{-\frac{n-2}{2}}.$$
Then, by the maximum principle,
$$S_u(x_0, t)\subset \{x\in V: g_V(x, x_0)>T\}.$$
It follows from the lower bound on the volume of sections in Lemma \ref{vol_lem} that
\begin{eqnarray*}c_1 t^{\frac{n}{2}}\leq |S_u(x_0, t)|&\leq& |\{x\in V: g_V(x, x_0)>T\}|\leq K(V, n,\lambda,\Lambda)T^{-\frac{n}{n-2}}\\ &<& K \left((\frac{K}{c_1})^{\frac{n-2}{n}} t^{-\frac{n-2}{2}}\right)^{-\frac{n}{n-2}}= c_1 t^{\frac{n}{2}}.
\end{eqnarray*}
This is a contradiction. Thus, we must have the desired upper bound.\\
{\bf 2.} The proof using Proposition \ref{Mthm} is similar to the above case and is thus omitted.
\end{proof}
\section{Capacity and lower bound for the Green's function}
\label{sec_n2}
In this section, we bound the Green's function using capacity in potential theory and give the proof for the lower bound of the Green's function in Theorem \ref{Gthm} (ii). 

Let $u$ be convex with compact sections and satisfies the Monge-Amp\`ere equation (\ref{MAeq}) with (\ref{muDC}).
Let $V$ be a fixed, open, bounded set in $\RR^n$ and let $K$ be a closed subset of $V$. 
We define the capacity of $K$ with respect to the linearized Monge-Amp\`ere operator $L_u:=-U^{ij}\p_{ij}$ and the set $V$ as the infimum of

$$Q_u (\Phi)=\int_{V} U^{ij} \Phi_i \Phi_j$$
among functions $\Phi\in H^{1}_{0}(V)$ satisfying $\Phi\geq 1$ on $K$. This infimum will be denoted by
$cap_{L_u}(K, V)$. 
In what follows, our arguments do not depend on the lower and upper bounds of the eigenvalues of the matrix $(U^{ij})$. Thus, when necessary, we can assume that $L_u$ is uniformly elliptic. In particular, we obtain as in \cite{LSW} the following theorem:
\begin{theorem} \label{cap_thm} Suppose that $S_u(x_0, 2t)\subset\subset V.$ Let $g_V$ be the Green's function for $L_u$ in $V$. Then there is a constant $C(n,\alpha, \beta)$ such that for all $x\in\p S_u(x_0, t)$
$$C^{-1}\left[cap_{L_u}(\overline{S_u(x_0, t)}, V)\right]^{-1}\leq g_V(x, x_0)\leq C \left[cap_{L_u}(\overline{S_u(x_0, t)}, V)\right]^{-1}.$$
\end{theorem}
\begin{proof}[Proof of the lower bound of the Green's function in Theorem \ref{Gthm}(ii)]
In view of Theorem \ref{cap_thm} and the maximum principle, the lower bound for the Green's function in Theorem \ref{Gthm}(ii) follows from the following capacity estimates:
$$cap_{L_u}(\overline{S_u(x_0, t)}, V)\leq \begin{cases} C(n,\alpha, \beta) \mu(S_u(x_0, t))t^{-1} &\mbox{if } n \geq 3 \\ \displaystyle
\frac{8}{ |\text{log t}|^2} \int_{t}^{t^{1/2}} \frac{\mu(S_u(x_0, s)) ds}{s^2}& \mbox{if } n = 2. \end{cases}$$
We will prove these estimates in Lemmas \ref{cap_3D_lem} and \ref{cap_2D_lem} below.
\end{proof}
\begin{lemma} \label{cap_3D_lem} Assume $n\geq 3$. Suppose that $S_u(x_0, 2t)\subset\subset V$. Then
$$cap_{L_u}(\overline{S_u(x_0, t)}, V)\leq C(n,\alpha, \beta) \mu(S_u(x_0, t))t^{-1}.$$
\end{lemma}
\begin{lemma} \label{cap_2D_lem} Assume $n=2$. Suppose that $S_u(x_0, t^{1/2})\subset\subset V$ and $0<t<1$. Then
\begin{equation}cap_{L_u}(\overline{S_u(x_0, t)}, V)\leq \frac{8}{ |\text{log t}|^2} \int_{t}^{t^{1/2}} \frac{\mu(S_u(x_0, s)) ds}{s^2}.
\label{upcap_2D}
\end{equation}
\end{lemma}
\begin{remark}
Lemma \ref{cap_3D_lem} can be deduced from the proof of \cite[Theorem 7.2]{MCV}. We present here a slightly different proof whose idea leads to the sharp bound for capacity in dimensions 2
in Lemma \ref{cap_2D_lem}. 
\end{remark}
We now prove Lemmas \ref{cap_3D_lem} and \ref{cap_2D_lem}. By subtracting a linear function, we can assume that
$u\geq 0, ~u(x_0)=0.$ Then $u=s$ on $\p S_u(x_0, s)$ for all $s>0$.
In the proofs of Lemmas \ref{cap_3D_lem} and \ref{cap_2D_lem}, we use the following general fact:
\begin{lemma}\label{bdrS}
We have
$$\int_{\p S_u(x_0, s)} U^{ij}\frac{u_i u_j}{|\nabla u|} = \int_{S_u(x_0, s)} n \det D^2 u.$$
\end{lemma}
\begin{proof}[Proof of Lemma \ref{bdrS}]
Let $\phi$ be any smooth function. Let $\nu= (\nu_1, \cdots,\nu_n)$ be the unit outer-normal to $\p S_u(x_0, s)$. Then, integrating by parts twice, and noting that $\nu = \frac{\nabla u}{|\nabla u|}$ on $\p S_u(x_0, s)$, we get
\begin{eqnarray*}
\int_{S_u(x_0, s)} (L_u \phi)u &=& \int_{S_u(x_0, s)} -U^{ij}\phi_{ij} u=\int_{S_u(x_0, s)} U^{ij}\phi_i u_{j}-\int_{\p S_u(x_0, s)} U^{ij}\phi_i u\nu_j\\&=& \int_{S_u(x_0, s)} -U^{ij} u_{ij} \phi +
\int_{\p S_u(x_0, s)} U^{ij} u_j \nu_i \phi -\int_{\p S_u(x_0, s)} U^{ij}\phi_i u\nu_j\\
&=& \int_{S_u(x_0, s)} -U^{ij} u_{ij} \phi + \int_{\p S_u(x_0, s)} U^{ij} \frac{u_i u_j}{|\nabla u|} \phi -\int_{\p S_u(x_0, s)} U^{ij}\frac{\phi_i u_j}{|\nabla u|} u.
\end{eqnarray*}
With $\phi\equiv 1$, using $U^{ij}u_{ij}= n\det D^2 u$, we obtain the equality claimed in the lemma.
\end{proof}
\begin{proof}[Proof of Lemma \ref{cap_3D_lem}]
Let us consider $h(x)=\gamma(u(x))$ where
$$\gamma(s)=\begin{cases} 1 &\mbox{if } s\leq t \\
 \frac{t^{\frac{n-2}{2}}}{1-(1/2)^{\frac{n-2}{2}}} (\frac{1}{s^{\frac{n-2}{2}}}-\frac{1}{(2 t)^{\frac{n-2}{2}}})& \mbox{if } t\leq s\leq 2t\\ 0 &\mbox{if } s\geq 2t. \end{cases}$$
Then $$h\in H^{1}_{0} (S_u(x_0, 2t))~\text{and}~h\equiv 1~\text{in}~ S_{u}(x_0, t).$$
We have
$$\nabla h(x) =\gamma'(u(x))\nabla u(x)= -\frac{n-2}{2}\frac{t^{\frac{n-2}{2}}}{1-(1/2)^{\frac{n-2}{2}}} u^{-\frac{n}{2}}\nabla u(x).$$
Therefore, by the coarea formula and Lemma \ref{bdrS}, we get
\begin{eqnarray*}
\int_{V} U^{ij} h_i h_j &=& \left[\frac{n-2}{2}\frac{t^{\frac{n-2}{2}}}{1-(1/2)^{\frac{n-2}{2}}}\right]^2 \int_{S_{u}(x_0, 2t)\backslash S_u(x_0, t)} U^{ij} \frac{u_i u_j}{u^n}
\\&\leq &C(n) t^{n-2}\int_{t}^{2t} \left(\int_{\p S_u(x_0, s)} U^{ij}\frac{u_i u_j}{s^n} \frac{1}{|\nabla u|}\right)ds\\
&=& C(n) t^{n-2}\int_{t}^{2t}\frac{\mu(S_u(x_0, s))}{s^n} ds
\\&\leq& C(n,\alpha,\beta)  \mu(S_u(x_0, t))t^{-1}
\end{eqnarray*}
where in the last inequality we used Lemma \ref{muDP} which says that
$$\mu(S_u(x_0, s))\leq C \mu(S_u(x_0, t))~\text{for}~t\leq s\leq 2t.$$
We now find from the definition of capacity that
\begin{eqnarray*}
cap_{L_u}(\overline{S_u(x_0, t)}, V)\leq \int_{V} U^{ij} h_i h_j \leq C(n,\alpha,\beta)  \mu(S_u(x_0, t))t^{-1}.
\end{eqnarray*}
\end{proof}

\begin{proof}[Proof of Lemma \ref{cap_2D_lem}]
Let us consider $h(x)=\gamma(u(x))$ where $\gamma$ is the logarithmic cut off function
$$\gamma(s)= \chi_{(-\infty, t)}(s)  + (2\text{log~} s/\text{log~} t-1)\chi_{[t, t^{1/2}]}(s). $$
Then $$h\in H^{1}_{0} (S_u(x_0, t^{1/2}))~\text{and}~h\equiv 1~\text{in}~ S_{u}(x_0, t).$$
We have
$$\nabla h(x) =\gamma'(u(x))\nabla u(x)= \frac{2}{ u~\text{log t}}\nabla u(x).$$
Therefore, by the coarea formula and Lemma \ref{bdrS}, we get
\begin{eqnarray*}
\int_{V} U^{ij} h_i h_j &=& \frac{4}{|\text{log t}|^2} \int_{S_{u}(x_0, t^{1/2})\backslash S_u(x_0, t)} U^{ij} \frac{u_i u_j}{u^2}\\&=&\frac{4}{|\text{log} t|^2}\int_{t}^{t^{1/2}} \left(\int_{\p S_u(x_0, s)} U^{ij}\frac{u_i u_j}{s^2} \frac{1}{|\nabla u|}\right)ds
 \\ &=&\frac{4}{|\text{log t}|^2}\int_{t}^{t^{1/2}}\left(\frac{1}{s^2} \int_{S_u (x_0, s)} n\det D^2 u\right) ds\\& =&
 \frac{8}{ |\text{log t}|^2} \int_{t}^{t^{1/2}} \frac{\mu(S_u(x_0, s)) ds}{s^2}.
\end{eqnarray*}
In the last equality, we used $n=2$. 
By the definition of capacity, we obtain (\ref{upcap_2D}).
\end{proof}
\begin{proof}[Sketch of proof of Theorem \ref{cap_thm}] We sketch here the proof of Theorem \ref{cap_thm}, following \cite{LSW}. We can assume that $L_u:= -U^{ij}\p_{ij}$ is uniformly elliptic. The set of functions $\Phi\in H^{1}_{0}(V)$ satisfying $\Phi\geq 1$ on $K$ is a closed convex set and $H_0^1(V)$ is a Hilbert space. It is then easy to see that there is a unique function $\Phi\in H^{1}_{0}(V)$ satisfying $\Phi\geq 1$ on $K$ and
$$cap_{L_u}(K, V)= Q_u (\Phi).$$ 
This function $\Phi$ is called the {\it capacitary potential} of the set $K$ with respect to the operator $L_u$ and the set $V$. Moreover, by a simple truncation argument, we find that this $\Phi$ satisfies $\Phi\equiv 1$ on $K$. 

The capacitary potential $\Phi$ of the compact set $K$ with respect to the operator $L_u$ and the set $V$ has the following properties:
\begin{myindentpar}{1cm}
(i) $\Phi\equiv 1$ on $K$,~ $\Phi=0$ on $\p V$, ~ $0\leq\Phi\leq 1$ on $V\backslash K.$\\
(ii) $L_u\Phi =0$ on $V\backslash K.$\\
(iii) For all $\varphi\in H^1_0(V)$ with $\varphi\geq 0$ on $K$, we have
$$\int_V U^{ij}\Phi_i\varphi_j\geq 0.$$ 
\end{myindentpar}
From (iii) and Schwartz's theorem on positive distributions, there is a nonnegative measure $\mu$ on $K$, called the {\it capacitary distribution} of $K$ with respect to the operator $L_u$ and the set $V$, such that
\begin{equation}\int_V U^{ij}\Phi_i\varphi_j=\int_V\varphi d\mu ~\text{for all } \varphi\in H^1_0(V) \text{ with } \varphi\geq 0 \text{ on } K.
 \label{cap_mu}
\end{equation}
Since $\Phi\equiv 1$ on $K$, the support of $\mu$ is on $\p K.$ Choosing $\varphi=\Phi$ in the above equation, we find that
\begin{equation}\mu(K)=cap_{L_u}(K, V).
\label{capmu}
\end{equation}
Moreover, we find from (\ref{cap_mu}) that $L_u\Phi=\mu~\text{in}~V.$ Thus, we have the representation
$$\Phi(y) =\int_{V} g_V(x, y) d \mu (x)$$
where we recall that $g_V(x, y)$ is the Green's function of $L_u$ in $V$.

Consider the set
$$J_a =\{x\in V: g_V(x, x_0)\geq a\}.$$
Let $\nu_a$ be the capacitary distribution of $J_a$ with respect to the operator $L_u$ and the set $V$. Then the capacitary potential of $J_a$ with respect to the operator $L_u$ and the set $V$ is equal to $1$ at $x_0$. Thus
$$1=\int_V g_V(x, x_0)d\nu_a(x).$$
The support of $\nu_a$ is on $\p J_a$ where $g_V(x, x_0)=a$. Thus, (\ref{capmu}) gives
$$cap_{L_u}(J_a, V)=\frac{1}{a}.$$
Let $a=\min_{x\in\p S_u(x_0, t)} g_V(x, x_0).$
Then, by the maximum principle
$\overline{S_u(x_0, t)}\subset J_a.$ Therefore
$$cap_{L_u}(\overline{S_u(x_0, t)}, V)\leq cap_{L_u}(\overline{J_a}, V)=\frac{1}{a}=\frac{1}{\min_{x\in\p S_u(x_0, t)} g_V(x, x_0)}.$$
Similarly, if we let $b=\max_{x\in\p S_u(x_0, t)} g_V(x, x_0)$. Then
 $$cap_{L_u}(\overline{S_u(x_0, t)}, V)\leq cap_{L_u}(\overline{J_b}, V)=\frac{1}{b}=\frac{1}{\max_{x\in\p S_u(x_0, t)} g_V(x, x_0)}.$$
It follows that
\begin{equation}\min_{x\in\p S_u(x_0, t)} g_V(x, x_0)\leq (cap_{L_u}(\overline{S_u(x_0, t)}, V))^{-1}\leq \max_{x\in\p S_u(x_0, t)} g_V(x, x_0). 
 \label{cap_S}
\end{equation}
Since $g_V(x, x_0)$ is a positive solution of $L_u g_V(\cdot, x_0)$ in $V\backslash \{x_0\}$, by Theorem \ref{MHolder_thm}, for each $t$ where $S_u(x_0, 2t)\subset\subset V$, we have
$$\max_{x\in\p S_u(x_0, t)} g_V(x, x_0)\leq \beta \min_{x\in\p S_u(x_0, t)} g(x, x_0).$$
This combined with (\ref{cap_S}) gives the desired conclusion.
\end{proof}

{\bf Acknowledgements.} The research of the author was supported in part by the National 
Science Foundation under grant DMS-1500400 and an Indiana University Summer Faculty Fellowship.

{} 
\end{document}